\RequirePackage{fix-cm}
\documentclass[smallextended]{svjour3}       
\smartqed  
\usepackage{graphicx}
\usepackage{algorithm}
\usepackage{amsmath}
\usepackage{amsfonts}              
\usepackage{epstopdf}
\usepackage{xcolor}
\usepackage{dsfont}
\usepackage{graphicx}
\usepackage{amsmath}
\usepackage{graphicx}
\usepackage{figsize}
\usepackage{subfig}
\begin{document}

\title{Numerical solution of fractional Fredholm integro-differential equations by spectral method with fractional basis functions
}
\titlerunning{Numerical solution of fractional Fredholm integro-differential equations...}

\author{Y. Talaei \and S. Noeiaghdam \and  H. Hosseinzadeh
}
\authorrunning{Y. Talaei et. al. }
\institute{Younes Talaei \at
          Faculty of Mathematical Sciences, University of Tabriz, Tabriz, Iran (Corresponding author).\\
              \email{ ytalaei@gmail.com, y$\_$talaei@tabrizu.ac.ir}
              \and
           Samad Noeiaghdam  \at
         Industrial Mathematics Laboratory, Baikal School of BRICS, Irkutsk National Research Technical
University, Irkutsk 664074, Russia. \\
Department of Applied Mathematics and Programming, South Ural State University, Lenin
prospect 76, Chelyabinsk 454080, Russia.
          \email{noiagdams@susu.ru}
           \and
              Hasan Hosseinzadeh \at
         Department of Mathematics, Ardabil Branch, Islamic Azad University, Ardabil 5615883895, Iran.
          \email{hasan$\_$hz2003@yahoo.com}
              }
\maketitle

\begin{abstract}
This paper presents an efficient spectral method for solving the fractional Fredholm integro-differential equations.
The non-smoothness of the solutions to such problems leads to the performance of spectral methods based on the classical polynomials such as Chebyshev, Legendre, Laguerre, etc, with a low order of convergence. For this reason, the development of classic numerical methods to solve such problems becomes a challenging issue.  Since the non-smooth solutions have the same asymptotic behavior with polynomials of fractional powers, therefore, fractional basis functions are the best candidate to overcome the drawbacks of the accuracy of the spectral methods.  On the other hand, the fractional integration of the fractional polynomials functions is in the class of fractional polynomials and this is one of the main advantages of using the fractional basis functions.
In this paper, an implicit spectral collocation method based on the fractional Chelyshkov basis functions is introduced.
The framework of the method is to reduce the problem into a nonlinear system of equations utilizing the spectral collocation method along with the fractional operational integration matrix. The obtained algebraic system is solved using Newton's iterative method.
Convergence analysis of the method is studied.  The numerical examples show the efficiency of the method on the problems with smooth and non-smooth solutions in comparison with other existing methods.

\end{abstract}
\section{Introduction}
Fredholm integro-differential equations appear in modeling various physical processes such as neutron transport problems \cite{Martin}, neural networks \cite{Jackiewicz}, population model \cite{Kemanci}, filtering and scattering problems \cite{Hale2}, inverse problems \cite{Beilina} and diseases spread \cite{Medlock}.
Fractional differential equations are important tools in the mathematical modeling of some real-phenomena problems with memory \cite{He}. Theoretical and numerical analysis of fractional differential equations has been considered by many researchers \cite{Kilbas,Bagley,Chow,Diethelm}. Consider the fractional Fredholm integro-differential equations of the form
\begin{equation}\label{two}
\left\lbrace
\begin{array}{lr}
D_{*0}^{\alpha}y(x)=g(x)+\displaystyle\int_{0}^{1}k(x,t)f(t,y(t))dt,\ \ \ \   0< \alpha < 1,\ x\in \Omega,\\
y(0)=c,
\end{array}\right.
\end{equation}
where  $ \Omega=[0,1] $, $c \in  \mathds{R}$, the function $g\in C( \Omega)$ and $k\in C( \Omega \times  \Omega)$ are known. The given function $f$ is  continuous and satisfy the following Lipschitz condition
argument; i.e.,
\begin{equation}\label{800}
\vert f(x,y_{2})-f(x,y_{2})\vert \leq L_{f} \vert y_{2}-y_{1}\vert,\ \ \ x\in  \Omega,\ \
\end{equation}
where $L_{f}$ is a positive constant. The operator $D_{*0}^{\alpha}$ denotes the Caputo fractional differential operator with $0<\alpha<1$ (see \cite{Diethelm}). There are some numerical methods to solve the problem (\ref{two}): CAS wavelet method \cite{Saeedi}, Chebyshev wavelet method \cite{Setia}, Alpert wavelet method \cite{Hag}, Adomian decomposition method \cite{Noor}, Taylor expansion method \cite{Zhao}, fractional Lagrange basis function\cite{Kumar}, rationalized Haar functions \cite{Rahimi} and Laguerre polynomials \cite{Bayram}.

The spectral collocation methods are efficient tools in numerical investigation of linear and non-linear problems. The exponential convergence of these methods are the main reason for its prominence in solving problems with smooth solutions compared to other numerical methods \cite{Canuto,21,15,H1,samad}.
The main purpose of these methods is to find an approximate solution in terms of finite number of basis functions so that the unknown coefficients determined by minimizing the error between the exact and approximate solutions.
 Orthogonal basis polynomials are the main tool in constructing approximate solutions in spectral methods. The solution of  fractional integral and/or differential equations may be non-smooth, as its derivatives may be infinite at the left end of the interval. For this reason, the efficiency of spectral methods in solving fractional problems is reduced by using standard basic polynomials such as Legendre, Chebyshev, Hermit, and so on. Therefore, decreasing the order of convergence is one of the main disadvantages in the implementation of these methods to solve fractional problems and as a challenging problem to overcome its drawbacks. Recently, some numerical methods have been introduced by modifying the standard basis polynomials by using the change of $x$ to $x^{\nu}$, $(0<\nu<1)$ \cite{9,Conte,Cai,Talaei2,Hale,Talaeim,Talaeid,Wang}. The existence, uniqueness, and smoothness of the solution of (\ref{two}) are investigated. The framework of this paper is to convert the problem into a fractional nonlinear integral equation and present a high-order implicit collocation method for its numerical solution. Because of  the non-smooth behavior of the solutions of (\ref{two}), we utilize the fractional Chelyshkov basis function of the form
\begin{equation*}
\widetilde{C}_{N,i}(x):=C_{N,i}(x^{\nu}),\ \ \  0<\nu<1.
\end{equation*}
where $\lbrace C_{N,i}(x)\rbrace_{i=0}^{N}$ are a set of orthogonal Chelyshkov polynomials on $[0,1]$ \cite{Chel}.
The advantage of the method is the reduction of the given problem into a algebraic system of equations. Simplicity in calculating the fractional operational matrix and high accuracy of the method in solving problems with non-smooth solutions by selecting a smaller number of fractional Chelyshkov polynomials are the other main advantage of our method in comparison with some numerical methods that uses basis functions such as Chebyshev-wavelet function \cite{Setia} and Laguerre polynomials \cite{Bayram}, CAS wavelet functions \cite{Saeedi}, Alpert multi-wavelets functions \cite{Hag}, rationalized Haar function \cite{Rahimi}, and fractional Lagrange polynomials \cite{Kumar}.

The layout of this paper included as follows: Section 2 contain the
fractional Chelyshkov polynomials are investigate and operational
matrices of integration are derived and also the existence, uniqueness and smoothness of the solution of problem (\ref{two}) is studied. In Section 3, the numerical method to solve the problem (\ref{two}) is constructed. Section 4 contain the
convergence analysis of the method. Numerical examples in Section 5 demonstrate the effectiveness of the method. Final section contain a
brief review of the paper.

\section{Fractional Chelyshkov polynomials}
In 2005, Vladimir S. Chelyshkov introduced a new type of orthogonal polynomials \cite{Chel}. These polynomials have been used based on spectral method to solve various type differential and integral equations \cite{Sezer,Talaei2,Izadi,Hosseininia,Talaei,Talaei1}.
The fractional Chelyshkov polynomials defined as follows
\begin{equation}\label{jadid}
\widetilde{C}_{N,n,\nu}(x)=\sum_{j=n}^{N}(-1)^{j-n}\binom{N-n}{j-n}\binom{N+j+1}{N-n}x^{j\nu},\ \ \ n=0,1,...,N.
\end{equation}
\begin{figure}[!ht]
\begin{center}$
\begin{array}{cc}
  \includegraphics[width=0.5\linewidth]{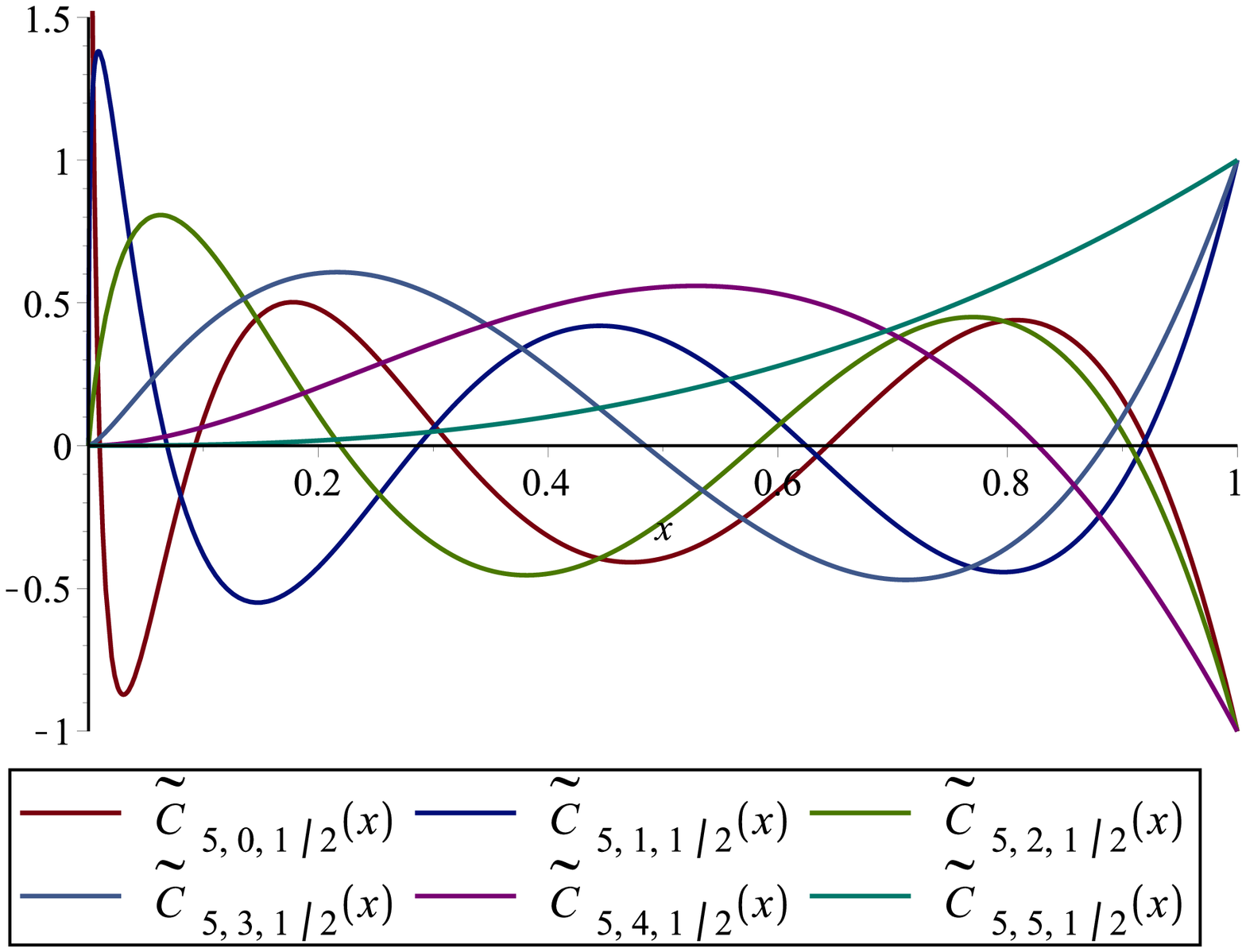}
  \includegraphics[width=0.5\linewidth]{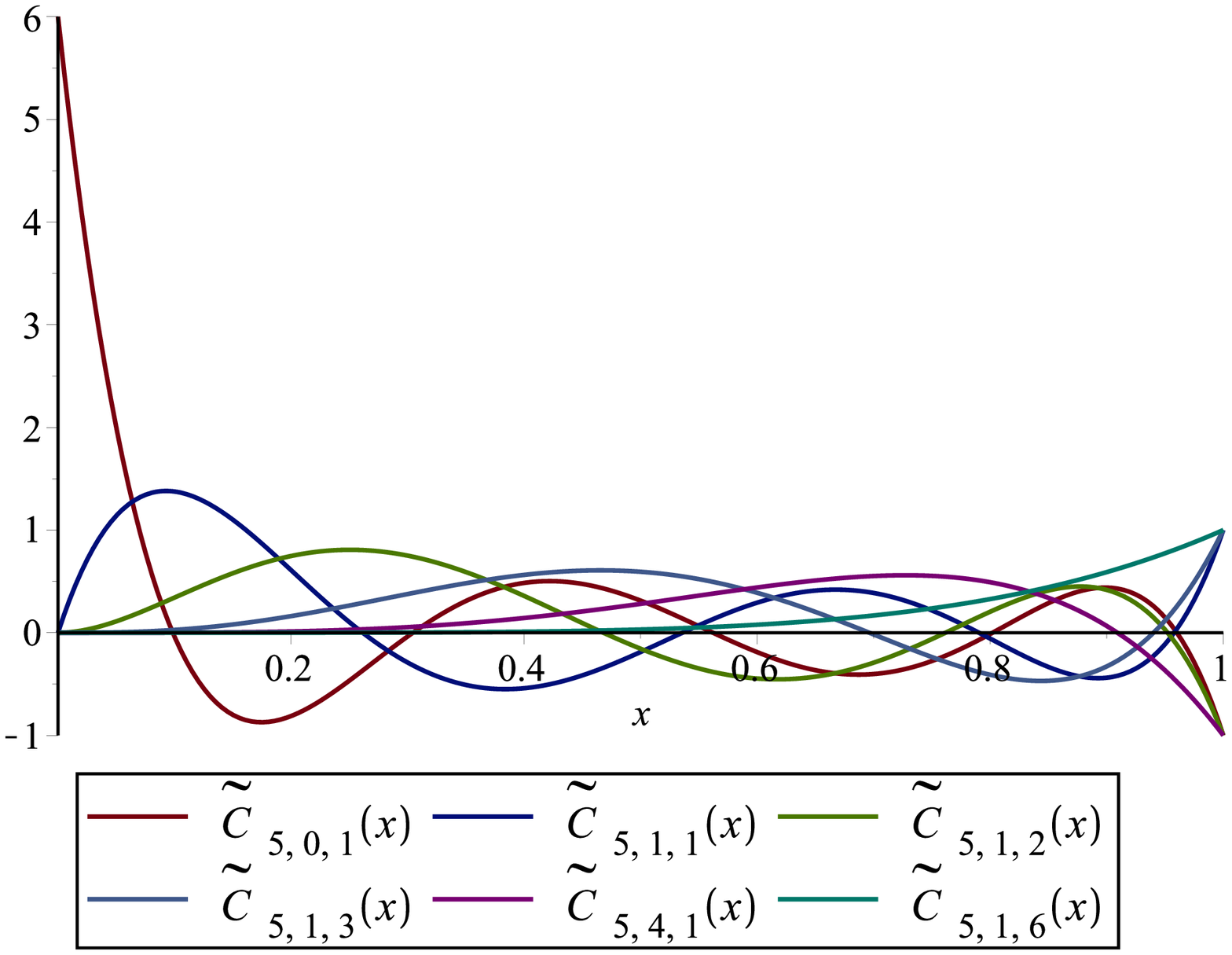}
\end{array}$
\end{center}
\caption{Plots of a fractional Chelyshkov basis: $\nu=1/2$ (left), $\nu=1$ (right) for $N=5$.}
\label{Aggn}
\end{figure}
These polynomials are connected to a set of fractional Jacobi polynomials \cite{9} $P^{a,b,\nu}_{n}(x)$, for $a,b>-1$  as
\[
\widetilde{C}_{N,n,\nu}(x)=(-1)^{N-n}x^{n\nu}P^{0,2n+1,\nu}_{N-n}(2x-1),
\]
and can be generated from the following recurrence relations
\cite{Chel}:
\begin{align}\label{ghgh}
&\widetilde{C}_{N,N,\nu}(x)=x^{N\nu},\ \ \ \widetilde{C}_{N,N-1,\nu}(x)=(2N)x^{(N-1)\nu}-(2N+1)x^{N\nu},\nonumber\\
&a_{N,k}\widetilde{C}_{N,k-1,\nu}(x)=(b_{N,k}x^{-\nu}-c_{N,k})\widetilde{C}_{N,k,\nu}(x)-d_{N,k}\widetilde{C}_{N,k+1,\nu}(x),
\end{align}
for $k=N-1,...,1$, where,
\begin{align}
&a_{N,k}=(k+1)(N-k+1)(N+k+1),\ \ \ b_{N,k}=k(2k+1)(2k+2),\nonumber\\
&c_{N,k}=(2k+1)((N+1)^2+k^2+k),\ \ \ d_{N,k}=k(N-k)(N+k+2).
\end{align}

For $N=5$, we have
\begin{align*}
&\widetilde{C}_{5,0,\nu}(x)=6-105\,{x}^{\nu}+560\,{x}^{2\,\nu}-1260\,{x}^{3\,\nu}+1260\,{
x}^{4\,\nu}-462\,{x}^{5\,\nu},\\
&\widetilde{C}_{5,1,\nu}(x)=35\,{x}^{\nu}-280\,{x}^{2\,\nu}+756\,{x}^{3\,\nu}-840\,{x}^{4
\,\nu}+330\,{x}^{5\,\nu},\\
&\widetilde{C}_{5,2,\nu}(x)=56\,{x}^{2\,\nu}-252\,{x}^{3\,\nu}+360\,{x}^{4\,\nu}-165\,{x}
^{5\,\nu},\\
&\widetilde{C}_{5,3,\nu}(x)=36\,{x}^{3\,\nu}-90\,{x}^{4\,\nu}+55\,{x}^{5\,\nu},\\
&\widetilde{C}_{5,4,\nu}(x)=10\,{x}^{4\,\nu}-11\,{x}^{5\,\nu},\\
&\widetilde{C}_{5,5,\nu}(x)={x}^{5\,\nu}.
\end{align*}
In Fig. \ref{Aggn}, the Chelyshkov polynomials $\lbrace \widetilde{C}_{5,0,\nu}(x),...,\widetilde{C}_{5,5,\nu}(x)\rbrace$ are plotted for $\nu=1,\frac{1}{2}$. All of $\widetilde{C}_{N,n,\nu}(x),\ n=0,...,N$ are polynomials exactly degree $N\nu$. The fractional Chelyshkov polynomials (\ref{jadid}) are orthogonal w.r.t weight function $\varpi(x)=x^{\nu-1}$;
 \begin{equation}\label{orto1}
\langle \widetilde{C}_{N,p,\nu}(x),\widetilde{C}_{N,q,\nu}(x)\rangle:=\int_{0}^{1}\widetilde{C}_{N,p,\nu}(x)\widetilde{C}_{N,q,\nu}(x)\varpi(x)dx=\left\lbrace
\begin{array}{lr}
 0,\ \ \ \ \ \ \ \ \ \ \ p\neq q,\\
\displaystyle\frac{1}{\nu(2p+1)}, \ \ \ \ p=q.
\end{array}\right.
 \end{equation}
Also,
\begin{equation}\label{jnnk}
\int_{0}^{1}\widetilde{C}_{N,n,\nu}(x)\varpi(x)dx=\frac{1}{\nu(N+1)}.
\end{equation}
Let
\[
L^{2}(\Omega)=\lbrace u:\Omega \rightarrow \mathds{R};\ \int_{\Omega}\vert u(x)\vert^{2}\varpi(x)dx <\infty \rbrace, \ \ \ \Vert u\Vert_{2}=\langle u,u\rangle^{\frac{1}{2}},
\]
be the weighted Hilbert space and
\[
M_{N}=Span\lbrace \widetilde{C}_{N,0,\nu}(x),\widetilde{C}_{N,1,\nu}(x),...,\widetilde{C}_{N,N,\nu}(x)\rbrace,
\]
be a finite-dimensional subspace of  $L^{2}(\Omega)$. The space $M_{N}$ is a complete subspace of $L^{2}(\Omega)$ and for any $u\in L^{2}(\Omega)$, there exist a unique best approximation $\widetilde{u}_{N}$ out of $M_{N}$ such that
\[
 \Vert u-\widehat{u}_{N} \Vert_{L^{2}} \leq \Vert u-v \Vert_{L^{2}}, \ \ \forall v \in M_{N},
\]
and implies that
\begin{equation}\label{bc}
\langle u-\widehat{u}_{N},\widetilde{C}_{N,n,\nu}(x)\rangle=0,\ \ n=0,...,N.
\end{equation}
Moreover, there exist unique coefficients $a_{0},a_{1},...,a_{N}$, such that
\begin{equation}\label{123}
\widehat{u}_{N}(x)=\sum_{n=0}^{N}a_{n}\widetilde{C}_{N,n,\nu}(x)=\mathbf{A}_{N}\mathbf{\Phi}(x),
\end{equation}
where,
\begin{equation*}
\mathbf{A}_{N}=[a_{0},a_{1},...,a_{N}],
\end{equation*}
and
\begin{equation}\label{wd}
\mathbf{\Phi}(x)=[\widetilde{C}_{N,0,\nu}(x),\widetilde{C}_{N,1,\nu}(x),...,\widetilde{C}_{N,N,\nu}(x)]^{T}.
\end{equation}
From the orthogonality condition (\ref{orto1}) and relation (\ref{bc}), we get
\begin{align*}
\int_{0}^{1}u(x)\widetilde{C}_{N,n,\nu}(x)\varpi(x)dx&=\int_{0}^{1}\widehat{u}_{N}(x)\widetilde{C}_{N,n,\nu}(x)
\varpi(x)dx\nonumber\\
&=\int_{0}^{1}\left(\sum_{i=0}^{N}a_{i}\widetilde{C}_{N,i,\nu}(x)\right)\widetilde{C}_{N,n,\nu}(x)\varpi(x)dx\nonumber\\
&=a_{n}\int_{0}^{1}\left(\widetilde{C}_{N,n,\nu}(x)\right)^{2}\varpi(x)dx,
\end{align*}
therefore,  the coefficients $a_{n}$ can be computed as
\begin{equation}\label{bvb}
a_{n}=\nu(2n+1)\int_{0}^{1}u(x)\widetilde{C}_{N,n,\nu}(x)\varpi(x)dx,\ \ \ n=0,1,...,N.
\end{equation}
\begin{lemma}\label{kl}\cite{Talaei1}
Let $x_{i}$ be the roots of $\widetilde{C}_{N,0,1}(x)$. Therefore, the fractional  Chelyshkov polynomial $\widetilde{C}_{N,0,\nu}(x)$ has $N$ roots as $x_{i}^{\frac{1}{\nu}}$ for $i=1,...,N$.
\end{lemma}

In the following theorem, we derive the fractional integration operational matrix of fractional Chelyshkov polynomials:
\begin{theorem}\label{plp}
Let $\mathbf{\Phi}(x)$ be the fractional Chelyshkov polynomials vector defined in (\ref{wd}). Then,
\[
\int _{0}^{x} (x-s)^{\alpha -1}\mathbf{\Phi}(s)ds \simeq \mathcal{P} \mathbf{\Phi}(x),
\]
where
\[
\mathcal{P}=\left(
  \begin{array}{cccc}
    \Theta(0,0) & \Theta(0,1) & \ldots & \Theta(0,0) \\
    \Theta(1,0) & \Theta(1,1) & \cdots & \Theta(1,N) \\
    \vdots & \vdots & \ddots & \vdots \\
    \Theta(N,0) & \Theta(N,1) & \ldots & \Theta(N,N) \\
  \end{array}
\right),
\]
with
\begin{equation}\label{Dad}
\Theta(n,k)=\sum_{j=n}^{N}(-1)^{j-n}\binom{N-n}{j-n}\binom{N+j+1}{N-n} B(\alpha , j\nu +1)\xi_{k,j},
\end{equation}
is called the fractional operational matrix of integration. Here, $B(\cdot,\cdot)$ denotes the Beta function.
\end{theorem}
\begin{proof}
Integrating of $\widetilde{C}_{N,n,\nu}(x)$ from $0$ to $x$ yields
\begin{equation}\label{treee}
\int_{0}^{x}(x-s)^{\alpha -1}\widetilde{C}_{N,n,\nu}(s)ds=\sum_{j=n}^{N}(-1)^{j-n}\binom{N-n}{j-n}\binom{N+j+1}{N-n} B(\alpha , j\nu+1)x^{j\nu+\alpha}.
\end{equation}
Approximating $x^{j\nu+\alpha}$ in terms of fractional Chelyshkov polynomials, we get
\begin{equation}\label{Za}
x^{j\nu+\alpha}\simeq \sum_{k=0}^{N}\xi_{k,j}\widetilde{C}_{N,k,\nu}(x),
\end{equation}
where the coefficients $\xi_{k,j}$ can be computed using (\ref{bvb}) as
\begin{align}\label{123}
\xi_{k,j}&=\nu(2k+1) \displaystyle\int_{0}^{1}x^{j\nu+\alpha}\widetilde{C}_{N,k,\nu}(x)\varpi(x)dx,\nonumber \\
&=\nu(2k+1) \displaystyle\sum_{l=k}^{N}(-1)^{l-k}\binom{N-k}{l-k}\binom{N+l+1}{N-k}\int_{0}^{1}x^{(j+l+1)\nu+\alpha-1}dx,\nonumber \\
&=\nu(2k+1)\displaystyle\sum_{l=k}^{N}\frac{(-1)^{l-k}}{(j+l+1)\nu+\alpha}\binom{N-k}{l-k}\binom{N+l+1}{N-k}.
\end{align}
By substituting (\ref{Za}), (\ref{123}) in (\ref{treee}), we have
\begin{align}\label{123t}
\int_{0}^{x}(x-s)^{\alpha -1}\widetilde{C}_{N,n,\nu}(s)ds\simeq\sum_{k=0}^{N}\Theta(n,k)\widetilde{C}_{N,k,\nu}(x).
\end{align}
Thus, the proof is completed.
\end{proof}
Now, we study the existence, uniqueness and smoothness of the solution to the problem (\ref{two}):
\begin{definition}\cite{Diethelm}
The Riemann--Liouville fractional integral of order $ \alpha $ for any $u\in L_{1}[a,b]$ is defined as
 \begin{equation}
J_{a}^{\alpha}u(x):=\frac{1}{\Gamma (\alpha)}\int_{a}^{x}(x-t)^{\alpha-1}u(t)dt.
 \end{equation}
For $\alpha=0$, we have  $J_{a}^{0}:=I$ the identity operator.\\
\end{definition}
\begin{definition}\cite{Diethelm}
 The operator $D_{*a}^{\alpha}$ defined by
\begin{equation}
D_{*a}^{\alpha}u(x):=J_{a}^{\lceil \alpha \rceil-\alpha}D^{\lceil \alpha \rceil}u(x)=\frac{1}{\Gamma (\lceil \alpha \rceil-\alpha)}\int_{a}^{x}(x-t)^{\lceil \alpha \rceil-\alpha-1}u^{\lceil \alpha \rceil}(t)dt
 \end{equation}
is called the  Caputo differential operator of order $\alpha\in \mathds{R}_{+}$.
\end{definition}

 \begin{theorem}\cite{fixed,Kantorovich} (Banach's Fixed Point Theorem)\label{A1}
Let $(\mathcal{U},d)$ to be a complete metric space and $ 0 \leq \gamma <1$. Also, assume that the function $ A:\mathcal{U} \rightarrow \mathcal{U}$ satisfy the inequality
\begin{equation}
 d(Au,Av)\leq \gamma d(u,v),
 \end{equation}
for every $u,v\in \mathcal{U}$. Then, there exists a unique $u^{*}$ such that $u^{*}=A(u^{*})$. Furthermore, for any $u_{0} \in \mathcal{U}$ we have
\[
A^{j}u_{0}\rightarrow u^{*},\ \ j\rightarrow \infty,
\]
that is the unique solution of the problem $u=Au$.
\begin{definition}\cite{Bru}
The space $C^{m, \lambda}(0,1]$, $m\in \mathds{N}$,  $-\infty<\lambda<1$, define a  set of all $m$ times continuously differentiable
functions $u:(0,1]\rightarrow \mathds{R}$ such that
\begin{equation}
\vert u^{(i)}(x)\vert \leq c_{i}
  \begin{cases}
   1, \quad i<1-\lambda, \\
      1+\vert \log(x)\vert, \quad i=1-\lambda, \\
    x^{1-\lambda-i}, \quad i>1-\lambda,
  \end{cases}
\end{equation}
holds with a constant $c_{i}$ for $i=0,...,m$.
\end{definition}
 \end{theorem}
By applying fractional integral operator $J_{0}^{\alpha}$ on both sides of the Eq. (\ref{two}), we obtain
\begin{align}\label{3}
 y(x)=\widetilde{g}(x)+ \int_{0}^{x}(x-t)^{\alpha-1}\widetilde{K}(y(t))dt,
\end{align}
where
\begin{equation}\label{has}
\widetilde{K}y(x)=\frac{1}{\Gamma (\alpha)}\int_{0}^{1}k(x,z)f(z,y(z))dz
\end{equation}
and
\begin{equation}\label{kok}
\widetilde{g}(x)=c+\frac{1}{\Gamma (\alpha)}\int_{0}^{x}(x-t)^{\alpha-1}g(t)dt.
\end{equation}
Therefore, the problem (\ref{two}) is equivalent to a fractional nonlinear Volterra integral equation of the form Eq. (\ref{3}). In the following, we consider the existence, uniqueness and smoothness of the solution of Eq. (\ref{3}).
\begin{theorem}\label{B} Assume that the function $f$ satisfy Lipschitz condition (\ref{800}) and
\[
\mathcal{M}_{k}:=\underset{x,t \in \Omega }{\max} \vert k(x,t)\vert.
\]
If
\begin{equation}\label{6}
 L_{f}\mathcal{M}_{k}<\Gamma (\alpha+1),
\end{equation}
 then, the Eq. (\ref{3}) has a unique continuous solution on $\Omega$.
\end{theorem}
\begin{proof}
Define the operator $A:C(\Omega)\rightarrow C(\Omega)$ as
 \begin{equation}\label{HGH}
 (Ay)(x):=\widetilde{g}(x)+ \int_{0}^{x}(x-t)^{\alpha-1}\widetilde{K}(y(t))dt.
\end{equation}
It is enough to prove that it has a unique fixed point. To this end, by using the Lipschitz condition for $f$ we obtain
\begin{align*}
 \vert(Ay)(x)-(A\widehat{y})(x)\vert &\leq \int_{0}^{x}(x-t)^{\alpha-1}\vert \widetilde{K}(y(t))-\widetilde{K}(\widehat{y}(t))\vert dt\\
 &\leq \frac{\mathcal{M}_{k}}{\Gamma (\alpha)} \int_{0}^{x}(x-t)^{\alpha-1}\left[\int_{0}^{1}\vert f(z,y(z))-f(z,\widehat{y}(z))\vert dz\right]dt\\
 &  \leq \frac{L_{f}\mathcal{M}_{k}}{\Gamma (\alpha)}\Vert y-\widehat{y} \Vert_{\infty} \int_{0}^{x}(x-t)^{\alpha-1} dt\\
&\leq   \frac{L_{f}\mathcal{M}_{k}}{\Gamma (\alpha+1)}\Vert y-\widehat{y}\Vert_{\infty},
\end{align*}
which implies that $A$ is a contraction mapping if only if
\[
\frac{L_{f}\mathcal{M}_{k}}{\Gamma (\alpha+1)}<1.
\]
From Theorem \ref{A1}, the operator $A$ defined in (\ref{HGH}) has a unique fixed point.
\end{proof}

\begin{theorem}\label{jko} Let $\widetilde{g}\in C^{m,1-\alpha}(\Omega)$,  $\widetilde{K}y\in C^{m}(\mathds{R})$,  $m \in \mathds{N}$ and $0<\alpha<1$. Then, the solution of Eq. (\ref{3}) satisfy the smoothness properties as follows
\begin{equation}
\vert y^{(i)}(x)\vert =O(x^{\alpha-i}),\ \ \ \  x\in (0,1],
\end{equation}
for $i=1,...,m$.
\end{theorem}

\begin{proof}
For the proof see Theorem 2.1 in \cite{Bru}.
\end{proof}

The result of Theorem \ref{jko} implies that the solution of Eq. (\ref{3}) has a singularity at the origin as  $t\rightarrow 0^{+}$ which indicates deterioration of the accuracy of many existing spectral methods based on standard basis functions such as Chebyshev, Legendre, etc. In order to overcome such drawbacks, we utilize the fractional Chelyshkov polynomials as basis functions to obtain the fractional order behavior and consistency between the approximate and exact solutions.
\section{Numerical method}
According to the implicit version of the collocation method in \cite{Sloan}, we consider the transformation
\begin{equation}\label{34}
w(x)=\widetilde{K}y(x),
\end{equation}
to approximate the solutions of Eq. (\ref{3}). By using (\ref{3}) and (\ref{has}), $w$ satisfy the following equation
\begin{align}\label{3vbbb}
w(x)=\frac{1}{\Gamma
(\alpha)}\int_{0}^{1}k(x,z)f\bigg(z,\widetilde{g}(z)+\displaystyle\int_{0}^{z}(z-t)^{\alpha-1}w(t)dt\bigg)dz.
\end{align}
After determining the unknown function $w$, then we can obtain a solution of the  Eq. (\ref{3}) by
\begin{align}\label{3vbbbv}
  y(x)=\widetilde{g}(x)+\int_{0}^{x}(x-t)^{\alpha-1}w(t)dt.
\end{align}
Now, we focus on the implementation of a spectral collocation method to solve Eq. (\ref{3vbbb}). To this end, we approximate $w(x)$ by the fractional Chelyshkov polynomials as
\begin{equation}\label{34}
w(x)\simeq w_{N}(x)=\sum_{i=0}^{N}w_{i}\widetilde{C}_{N,i,\nu}(x)=\mathbf{W}_{N}\mathbf{\Phi}(x),
\end{equation}
where, $\mathbf{W}_{N}=[w_{0},...,w_{N}]$ is an unknown vector.
\begin{theorem}\label{popp}
Let $w_{N}$ be the approximation defined in (\ref{34}) and $\mathcal{P}$ be fractional operational matrix defined in Theorem \ref{plp}, then we have
\begin{itemize}
\item[(A)] $\displaystyle\int_{0}^{x}(x-t)^{\alpha-1}w_{N}(t)dt\simeq \widehat{\mathbf{W}}_{N}\mathbf{\Phi}(x)$,\\
\item[(B)] $\widetilde{g}(x)\simeq (\mathcal{C}+\mathcal{G})\mathbf{\Phi}(x)$,
\end{itemize}
in which $\mathcal{C}:=[\mathcal{C}_{0},...,\mathcal{C}_{N}]$, $\widetilde{\mathbf{G}}:=[g_{0},...,g_{N}]$,
with
\[
\mathcal{C}_{i}=\frac{c(2i+1)}{N+1},\ \ g_{i}=\frac{\nu(2i+1)}{\Gamma(\alpha)} \int_{0}^{1} \widetilde{C}_{N,i,\nu}(x)g(x)\varpi(x)dx,
\]
for $i=0,...,N$, and $\widehat{\mathbf{W}}_{N}=\mathbf{W}_{N} \mathcal{P}$, $\mathcal{G}=\widetilde{\mathbf{G}}\mathcal{P}$.
\end{theorem}
\begin{proof}
Part (A): \\
From Theorem \ref{plp}, we have
 \begin{align}\label{vol}
 \int_{0}^{z}(z-t)^{\alpha-1}w_{N}(t)dt&=\mathbf{W}_{N}\int_{0}^{z}(z-t)^{\alpha-1}\mathbf{\Phi}(t)dt \nonumber\\
 &\simeq \mathbf{W}_{N} \mathcal{P} \mathbf{\Phi}(x)\nonumber\\
 &=\widehat{\mathbf{W}}_{N}\mathbf{\Phi}(z).
 \end{align}
Part (B):\\
Consider to (\ref{kok}), and let $\frac{1}{\Gamma(\alpha)}g(x)\simeq
\widetilde{\mathbf{G}}\mathbf{\Phi}(x)$,
$c=\mathcal{C}\mathbf{\Phi}(x)$, from (\ref{jnnk}) and (\ref{bvb}), we get
\[
\mathcal{C}_{i}=\nu(2i+1) \int_{0}^{1} c\
\widetilde{C}_{N,i,\nu}(x) \varpi(x)dx=\frac{c(2i+1)}{N+1},\ \
\]
and
\[
g_{i}=\frac{\nu(2i+1)}{\Gamma(\alpha)} \int_{0}^{1} \widetilde{C}_{N,i,\nu}(x)g(x)\varpi(x)dx.
\]
By using Theorem \ref{plp}, we can
write
\begin{align*}
\widetilde{g}(x)&=c+\frac{1}{\Gamma (\alpha)}\int_{0}^{x}(x-t)^{\alpha-1}g(t)dt\nonumber\\
&=\mathcal{C}\mathbf{\Phi}(x)+\int_{0}^{x}(x-t)^{\alpha-1}\widetilde{\mathbf{G}}\mathbf{\Phi}(t)dt\nonumber\\
& \simeq \mathcal{C}\mathbf{\Phi}(x)+\widetilde{\mathbf{G}}\int_{0}^{x}(x-t)^{\alpha-1}\mathbf{\Phi}(t)dt\nonumber\\
&=\left(\mathcal{C}+\widetilde{\mathbf{G}}\mathcal{P}\right) \mathbf{\Phi}(x)\nonumber\\
&=\left(\mathcal{C}+\mathcal{G}\right) \mathbf{\Phi}(x),
\end{align*}
which completes the proof.
\end{proof}
Assume that $\ell^{\nu}_{i}(x)$ be the fractional Lagrange basis function associated with the points $\widehat{x}_{i}=x_{i}^{\frac{1}{\nu}},\ i=0,...,N$, the roots of the polynomial $\widetilde{C}_{N+1,0}(x)$,  and define the fractional interpolation operator as
\begin{equation}\label{daroon}
\mathcal{I}_{N}u(x)=\sum_{i=0}^{N}u(\widehat{x}_{i})\ell^{\nu}_{i}(x).
\end{equation}
By substituting $w_{N}(x)$ in Eq. (\ref{3vbbb}) and applying $\mathcal{I}_{N}$ we obtain
\begin{equation}\label{pgj}
\mathcal{I}_{N}w_{N}(x)=\mathcal{I}_{N}\bigg(\frac{1}{\Gamma
(\alpha)}\int_{0}^{1}k(x,z)f\bigg(z,\widetilde{g}(z)+\displaystyle\int_{0}^{z}(z-t)^{\alpha-1}w_{N}(t)dt\bigg)dz\bigg),
\end{equation}
consequently,
\begin{align}\label{3vvvd}
w_{N}(\widehat{x}_{i})=\frac{1}{\Gamma
(\alpha)}\int_{0}^{1}k(\widehat{x}_{i},z)f\bigg(z,\widetilde{g}(z)+\displaystyle\int_{0}^{z}(z-t)^{\alpha-1}w_{N}(t)dt\bigg)dz.
\end{align}
By using Theorem \ref{popp} in Eq. (\ref{3vvvd}), we obtain
\begin{equation}\label{qq}
\mathbf{W}_{N}\mathbf{\Phi}(\widehat{x}_{i})=\frac{1}{\Gamma (\alpha)}\int_{0}^{1}k(\widehat{x}_{i},z)f\bigg(z,\big(\mathcal{C}+\mathcal{G}+\widehat{\mathbf{W}}_{N}\big)\mathbf{\Phi}(z)\bigg)dz.
\end{equation}
The integral term in (\ref{qq}) approximated by Gauss-Legendre quadrature formula \cite{Canuto} on $[0,1]$ with the weights and nodes
$(z_{\ell},\omega_{\ell})_{\ell=0}^{N}$,
\begin{equation}\label{908}
\int_{0}^{1}k(\widehat{x}_{i},z)\mathcal{H}(z)dz
\simeq
\sum_{\ell=0}^{N}\omega_{\ell}k(\widehat{x}_{i},z_{\ell})\mathcal{H}(z_{\ell}),
\end{equation}
where $\mathcal{H}(z)=\displaystyle\frac{1}{\Gamma (\alpha)}f\bigg(z,\big(\mathcal{C}+\mathcal{G}+\widehat{\mathbf{W}}_{N}\big)\mathbf{\Phi}(z)\bigg)$.  Then, substituting  Eq. (\ref{908}) in Eq. (\ref{qq}), we obtain
\begin{equation}\label{up}
\mathbf{f}_{i}(\mathbf{W}_{N})=\mathbf{W}_{N}\mathbf{\Phi}(\widehat{x}_{i})-\sum_{\ell=0}^{N}\omega_{\ell}k(\widehat{x}_{i},z_{\ell})\mathcal{H}(z_{\ell})=0.
\end{equation}
Therefore,
\begin{align}\label{oppo}
 \mathds{F}_{N}(\mathbf{W}_{N})&=\big[ \mathbf{f}_{0}(\mathbf{W}_{N}),...,\mathbf{f}_{N}(\mathbf{W}_{N})\big]\equiv 0,
\end{align}
which gives a nonlinear algebraic system that can be solved by Newton's iterative method. The approximate solution for the Eq. (\ref{3}) is obtained in the following form
\begin{align}\label{56}
  y_{N}(x)&= \left(\mathcal{C}+\mathcal{G}+\mathbf{W}_{N} \mathcal{P}\right) \mathbf{\Phi}(x).
\end{align}
Newton's iterative method reads as follows:
\begin{align}\label{dcdc}
\left\{
  \begin{array}{ll}
   \mathds{J}(\mathbf{W}_{N,i})\delta_{N,i}=-\mathds{F}_{N}(\mathbf{W}_{N,i}); \\
    \mathbf{W}_{N,i+1}\leftarrow\mathbf{W}_{N,i}+\delta_{N,i},\\
    i\leftarrow i+1,
  \end{array}
\right.
\end{align}
with initial guess $\mathbf{W}_{N,0}$ and end condition $\Vert \mathds{F}_{N}(\mathbf{W}_{N,i})\Vert \leq \epsilon$,
where be a small enough number. The Jacobian matrix
$ \mathds{J}$ is defined as
\[
\mathds{J}_{i,j}=\frac{\partial \mathbf{f}_{i} }{\partial w_{j}}.
\]
By applying the iterative process (\ref{dcdc}), a sequence of approximate solution
\[
w_{N,i}(x)=\mathbf{W}_{N,i}\mathbf{\Phi}(x),\ i=0,1,2,...,
\]
is generated. It can be seen that for $\Vert w_{N,i}-w_{N}\Vert \rightarrow 0$, the Jacobian matrix should be nonsingular.
In the next section, we state some convergence results for Newton's method. To select a proper initial guess for Newton's methods, by using the initial condition $y_{N}(0)=c=\mathcal{C} \mathbf{\Phi}(0)$ and Eq. (\ref{56}), we choose the initial guess such that
\[
y_{N}(0)=\left(\mathcal{C}+\mathcal{G}+\mathbf{W}_{N,0} \mathcal{P}\right) \mathbf{\Phi}(0)=\mathcal{C} \mathbf{\Phi}(0).
\]
Since $\mathcal{G}=\widetilde{\mathbf{G}}\mathcal{P}$, we conclude that $\mathbf{W}_{N,0}=-\widetilde{\mathbf{G}}$.
\section{Convergence analysis}
In this section, we obtain an upper bound for the error vector of the fractional integration operational matrix and analyze the convergence of the method.
\begin{theorem}\label{pkl}(Generalized Taylor series \cite{Odibat})
 Let
\[
D_{*,0}^{i\nu}u(x)\in C(0,1], \ i=0,...,N+1,
\]
where $0<\theta< 1$. Then,
\begin{equation}\label{1231}
u(x)=\sum_{i=0}^{N}\frac{D_{*,0}^{i\nu}u(0)}{\Gamma(i\nu+1)}x^{i\nu}+\frac{x^{(N+1)\nu}}{\Gamma((N+1)\nu+1)}D_{*,0}^{(N+1)\nu}u(x)\vert_{x=\xi},
\end{equation}
with $0 < \xi \leq x$, $\forall x\in (0,1]$.
\end{theorem}

\begin{theorem}\label{098}
Let $D_{*,0}^{i\nu}u(x)\in C(0,1], \ i=0,...,N+1$, $0<\nu< 1$ and $\widehat{u}_{N}(x)=\sum_{n=0}^{N}a_{n}\widetilde{C}_{N,n,\nu}(x)$  be the best approximation to $u(x)$ out of $M_{N}$. Then,
\begin{equation}\label{bon}
\Vert u-\widehat{u}_{N}\Vert_{2} \leq \frac{\mathcal{N}_{\nu}}{\Gamma((N+1)\nu+1)\sqrt{(2N+3)\nu}},
\end{equation}
in which $\mathcal{N}_{\nu}:=\underset{x\in [0,1]}{\max} \vert D_{*,0}^{(N+1)\nu}u(x) \vert $.
\end{theorem}
\begin{proof}
From Theorem \ref{pkl}, we have
\begin{align}
 \Vert u-\widehat{u}_{N} \Vert^{2} _{2} \leq \Vert u-\sum_{i=0}^{N}\frac{D_{*,0}^{i\nu}u(0)}{\Gamma(i\nu+1)}x^{i\nu}\Vert^{2} _{2}
 &\leq\int_{0}^{1}\bigg(  \frac{\mathcal{N}_{\nu}}{\Gamma((N+1)\nu+1)}x^{\nu-1} \bigg)^{2}dx\nonumber\\
  &\leq \bigg(\frac{\mathcal{N}_{\nu}}{\Gamma((N+1)\nu+1)}\bigg)^{2} \int_{0}^{1}x^{2(N+1)\nu}x^{\nu-1}dx\nonumber\\
   &\leq \bigg(\frac{\mathcal{N}_{\nu}}{\Gamma((N+1)\nu+1)}\bigg)^{2}\frac{1}{(2N+3)\nu}.
\end{align}
\end{proof}

\begin{corollary}\label{9877}
For the best approximate solution $\widehat{u}_{N}(x)=\sum_{n=0}^{N}a_{n}\widetilde{C}_{N,n,\nu}(x)$ to $u(x)$, from Theorem \ref{098} we have the following error bound
\begin{equation}\label{bon1}
\Vert u-\widehat{u}_{N}\Vert_{2}=O\left(\frac{1}{\Gamma((N+1)\nu+1)\sqrt{(2N+3)\nu}}\right).
\end{equation}
\end{corollary}

\begin{theorem}\label{0123}
\cite{Kreyszig}
Assume the hypothesis of Theorem \ref{098} is held. Then,
\[
\Vert u-\widehat{u}_{N} \Vert_{2}^{2}=\frac{\Psi(u,\widetilde{C}_{N,0,\nu},\widetilde{C}_{N,1,\nu},...,\widetilde{C}_{N,N,\nu})}{\Psi(\widetilde{C}_{N,0,\nu},\widetilde{C}_{N,1,\nu},...,\widetilde{C}_{N,N,\nu})},
\]
where
\[
\Psi(u,\phi_{1},\phi_{2},...,\phi_{N}):=\left|
  \begin{array}{cccc}
    \langle u,u\rangle & \langle u,\phi_{1}\rangle & \ldots & \langle u,\phi_{N}\rangle \\
   \langle \phi_{1},u\rangle & \langle \phi_{1},\phi_{1}\rangle & \ldots & \langle \phi_{1},\phi_{N}\rangle \\
    \vdots &  \vdots &\vdots & \vdots \\
  \langle \phi_{N},u\rangle & \langle \phi_{N},\phi_{1}\rangle & \ldots & \langle \phi_{N},\phi_{N}\rangle \\
  \end{array}
\right|.
\]
\end{theorem}
\begin{theorem}\label{KKK}Let
\begin{equation*}\label{E0}
\mathcal{E}(x)=[e_{0}(x),...,e_{N}(x)]:=\int _{0}^{x} (x-s)^{\alpha -1}\mathbf{\Phi}(s)ds- \mathcal{P} \mathbf{\Phi}(x)
\end{equation*}
 be the error vector related to $\mathcal{P}$. Then,
\begin{equation}
 \Vert e_{n}\Vert_{2} \rightarrow 0, \ \ \ \ \  N\rightarrow \infty,
\end{equation}
for $n=0,...,N$.
\end{theorem}
\begin{proof}
From (\ref{treee}) and (\ref{Za}), we have
\begin{align}\label{xvv}
e_{n}(x)=\sum_{j=n}^{N}(-1)^{j-n}\binom{N-n}{j-n}\binom{N+j+1}{N-n} B(\alpha , j\nu +1)\left( x^{j\nu+\alpha}- \sum_{k=0}^{N}\xi_{k,j}\widetilde{C}_{N,k}(x)\right),
\end{align}
for $n=0,1,...,N$. On the other hand, from Theorem \ref{0123}, we can get
\begin{equation}\label{kgg}
\Vert x^{j\nu+\alpha}- \sum_{k=0}^{N}\xi_{k,j}\widetilde{C}_{N,k}(x) \Vert_{2}^{2}\leq
\frac{\Psi(x^{j\nu+\alpha},\widetilde{C}_{N,0,\nu},\widetilde{C}_{N,1,\nu},...,\widetilde{C}_{N,N,\nu})}{\Psi(\widetilde{C}_{N,0,\nu},\widetilde{C}_{N,1,\nu},...,\widetilde{C}_{N,N,\nu})},
\end{equation}
therefore,
\begin{align}
\Vert e_{n}\Vert_{2}&\leq \sum_{j=n}^{N}\binom{N-n}{j-n}\binom{N+j+1}{N-n} B(\alpha , j\nu +1)\left(\frac{\Psi(x^{j\nu+\alpha},\widetilde{C}_{N,0,\nu},\widetilde{C}_{N,1,\nu},...,\widetilde{C}_{N,N,\nu})}{\Psi(\widetilde{C}_{N,0,\nu},\widetilde{C}_{N,1,\nu},...,\widetilde{C}_{N,N,\nu})}\right)^{1/2},
\end{align}
that gives an upper bound for each component of the error vector. Finally, from  (\ref{kgg}) and Corollary \ref{9877}, we can conclude that
\begin{align*}
\Vert e_{n}\Vert_{2}&\rightarrow 0,\ \ \ \ N\rightarrow \infty,
\end{align*}
that complete the proof.
\end{proof}

For example, let $N=4$ and $\alpha=\nu=\frac{1}{2}$, we have
\begin{align*}
&\Vert e_{0}\Vert_{2}\leq 1.5666\times 10^{-1},\ \ \Vert e_{1}\Vert_{2}\leq 9.3999\times 10^{-2},\\\ \
&\Vert e_{2}\Vert_{2}\leq 3.1333\times 10^{-2},\ \ \ \Vert e_{3}\Vert_{2}\leq 4.4761\times 10^{-3}.\ \
\end{align*}

\begin{theorem}\label{B} Assume that $y(x)$ and $y_{N}(x)$ are the exact solution and approximate solution of \eqref{two}, respectively. Then,
 \begin{equation}\label{bbm}
\Vert y-y_{N}\Vert_{2}\rightarrow 0,\ \ \ N\rightarrow \infty.
\end{equation}
\end{theorem}
\begin{proof}
By subtracting \eqref{56} from \eqref{3vbbbv}, we have
\begin{align}
  y(x)-y_{N}(x)=\widetilde{g}(x)-\big(\mathcal{C}+\mathcal{G}\big)\mathbf{\Phi}(x)+\int_{0}^{x}(x-t)^{\alpha-1}\left(w(t)-w_{N}(t)\right)dt,
\end{align}
 that yields
\begin{equation}
\Vert y-y_{N}\Vert_{2} \leq \Vert \widetilde{g}(x)-\big(\mathcal{C}+\mathcal{G}\big)\mathbf{\Phi}(x)\Vert_{2} +\Vert \int_{0}^{x}(x-t)^{\alpha-1}\big(w(t)-w_{N}(t)\big)dt\Vert_{2}.
\end{equation}
On the other hand,
\begin{align}
w(x)-w_{N}(x)&=\frac{1}{\Gamma
(\alpha)}\int_{0}^{1}k(x,z)f\bigg(z,\widetilde{g}(z)+\displaystyle\int_{0}^{z}(z-t)^{\alpha-1}w(t)dt\bigg)dz\nonumber\\
&-\frac{1}{\Gamma (\alpha)}\int_{0}^{1}k(x,z)f\bigg(z,\big(\mathcal{C}+\mathcal{G}+\widehat{\mathbf{W}}_{N}\big)\mathbf{\Phi}(z)\bigg)dz,
\end{align}
and from (\ref{800}), we can write
\begin{align}\label{4rf}
\vert w(x)-w_{N}(x)\vert &\leq \frac{L_{f}\mathcal{M}_{k}}{\Gamma
(\alpha)}\bigg(\int_{0}^{1}\vert \widetilde{g}(z)-\big(\mathcal{C}+\mathcal{G}\big)\mathbf{\Phi}(z)\vert dz\nonumber\\
&+\int_{0}^{1}\vert \displaystyle\int_{0}^{z}(z-t)^{\alpha-1}w(t)dt-\widehat{\mathbf{W}}_{N}\mathbf{\Phi}(z)\vert dz\bigg).
\end{align}
So, by using Theorems \ref{098} and \ref{KKK} in (\ref{4rf}), we can obtain the desired result (\ref{bbm}).
\end{proof}

Now, we discuss the conditions under which Newton's method is convergent. To this end, we consider the operator form of Eq. (\ref{oppo}) as follows
\begin{equation}\label{ppqq}
\mathcal{F}_{N}(w_{N})=w_{N}-\mathcal{I}_{N}\mathcal{K}_{N}w_{N}\equiv 0,
\end{equation}
where $\mathcal{K}_{N}$ is an approximate quadrature of integral operator $\mathcal{K}$ defined as
\[
\mathcal{K}w(x)=\frac{1}{\Gamma(\alpha)}\int_{0}^{1}k(x,z)f\bigg(z,\widetilde{g}(z)+\displaystyle\int_{0}^{z}(z-t)^{\alpha-1}w(t)dt\bigg)dz.
\]
The Frechet derivative of $\mathcal{F}_{N}$ at $w_{N}$ is defined as
\[
\mathcal{F}'_{N}(w_{N})(v)=v-\mathcal{I}_{N}\mathcal{K}'_{N}(w_{N})(v),
\]
in which
\[
\mathcal{K}'_{N}(w)(v)=\displaystyle\frac{1}{\Gamma (\alpha)}\sum_{\ell=0}^{N}\omega_{\ell}k(x,z_{\ell})f'\bigg(z,\big(\mathcal{C}+\mathcal{G}+\widehat{\mathbf{W}}_{N}\big)\mathbf{\Phi}(z_{\ell})\bigg)v(z_{\ell}),
\]
and $ f':=f_{y}(x,y)\in C(\Omega)$. From Lemma 2.2 in \cite{Anselone}, we can conclude that if
\[
\Vert \mathcal{I}_{N}\mathcal{K}'_{N}w_{N}-\mathcal{K}'w\Vert \rightarrow 0 ,\ \ \ N \rightarrow \infty,
\]
and $\mathcal{K}'w$ has no eigenvalue equal to 1, then $[I-\mathcal{I}_{N}\mathcal{K}'_{N}w_{N}]$ is invertible. To this end,
assume the following conditions hold:
\begin{itemize}
\item[(R1)]
$
\vert f_{y}(x,y)-f_{y}(x',y)\vert \leq C_{1} \vert x-x'\vert^{\beta},
$
\item[(R2)]
$
\vert f_{y}(x,u)-f_{y}(x,v)\vert \leq C_{2} \vert u-v\vert,
$
\end{itemize}
where $C_{i}$ are positive constants. According to triangular inequality, we get
\begin{align}
\Vert \mathcal{I}_{N}\mathcal{K}'_{N}w_{N}-\mathcal{K}'w\Vert_{2} &\leq  \Vert \mathcal{I}_{N}\mathcal{K}'w-\mathcal{K}'w\Vert_{2}+\Vert \mathcal{I}_{N}\mathcal{K}'_{N}w_{N}-\mathcal{I}_{N}\mathcal{K}'_{N}w\Vert_{2} \nonumber\\
&+\Vert \mathcal{I}_{N}\mathcal{K}'_{N}w-\mathcal{I}_{N}\mathcal{K}'w\Vert_{2}.
\end{align}
From Corollary (\ref{9877}) and that $\mathcal{I}_{N}\mathcal{K}'w \in M_{N}$ is the best approximation to $\mathcal{K}'w$, under condition (R1), we conclude that
\[
\Vert \mathcal{I}_{N}\mathcal{K}'w-\mathcal{K}'w\Vert_{2}  \rightarrow 0,\ \ \ N \rightarrow \infty.
\]
Using condition (R2) and (\ref{4rf}), we have
\[
\Vert \mathcal{I}_{N}\mathcal{K}'_{N}w_{N}-\mathcal{I}_{N}\mathcal{K}'_{N}w\Vert_{2}\leq \Vert \mathcal{I}_{N}\Vert_{2} \Vert w_{N}-w\Vert_{2}
\rightarrow 0,
\]
as $N \rightarrow \infty$. From Theorem 1 in \cite{Nevai} and integration error estimation from Gauss-quadrature role (\cite{Canuto}, p. 290), one can show that
\[
\Vert \mathcal{I}_{N}\mathcal{K}'_{N}w-\mathcal{I}_{N}\mathcal{K}'w\Vert_{2} \rightarrow 0,\ \ \ N \rightarrow \infty.
\]
In the following theorem, we deal with the local convergence of Newton's method:
\begin{theorem}
Assume that $w_{N}$ is the solution of Eq. (\ref{ppqq}) and $[I-\mathcal{K}'w]^{-1}$ exists. Assume further that the conditions (R1)-(R2) be held. Then, there exist a $\epsilon>0$ such that if $\Vert w_{N,0}-w_{N}\Vert \leq \epsilon$, then Newton's method converges. Furthermore,
\[
\Vert w_{N,i}-w_{N}\Vert \leq  \frac{(r\epsilon)^{2^{i}}}{r},
\]
provided that $r\epsilon<1$ for some constant $r$.
\end{theorem}
\begin{proof}
If 1 is not the eigenvalue of $\mathcal{K}'w$, then $[I-\mathcal{K}'w]$ is  invertible. The proof is straightforward from Theorem 5.4.1 in \cite{Atkinson} and the above discussion.
\end{proof}
\section{Numerical examples}
In this section, we intend to show the accuracy of the proposed method to solve the problem (\ref{two}) with smooth and non-smooth solutions. All calculations are computed by Maple 2018 with Digits=40. The $L_{2}$-error is measured in the following way
\[
\Vert E_{N}\Vert_{2}=\bigg(\sum_{l=0}^{N}\vert E_{N}(x_{l}) \vert^{2} \bigg)^{1/2} \approx
\sqrt{\frac{\sum_{l=0}^{N}\vert _{N}(x_{l}) \vert^{2}}{N}},
\]
where $E_{N}(x)=\vert y(x)-y_{N}(x) \vert$, denotes the absolute difference between exact and approximate solution.
In these examples, $m$ denotes the number of Newton's iterations with initial value $W_{N,0}=-\widetilde{\mathbf{G}}$.

 The steps of the numerical method can be summarized as follows:
\\
\textbf{Input:}
Input $N$, $\alpha$, $\nu$, $k$, $f$ and $g$.
\smallskip \\
\textbf{Output:} The approximate solution $y_{N}(x)=\widetilde{g}(x)+\mathbf{W} \mathcal{P} \mathbf{\Phi}(x)$.
\smallskip \\
\textbf{Step 1.} Construct the vector basis $\mathbf{\Phi}(t)$ in (\ref{wd}) from (\ref{ghgh}).
\smallskip  \\
\textbf{Step 2.} Compute the vectors $\mathcal{C},\mathcal{G},\widehat{\mathbf{W}}$ from Theorem \ref{popp}.
\smallskip  \\
\textbf{Step 3.} Construct the nonlinear algebraic system (\ref{oppo}) using
collocation points $\widehat{x}_{i},\ i=0,...,N$ and quadrature
formulae with weights and nodes
$(z_{\ell},\omega_{\ell})_{\ell=0}^{N}$.
\smallskip  \\
\textbf{Step 4.} Solve the system (\ref{oppo}) using Newton's iterative method.

\begin{example}\label{DWr}
Consider the problem
\[
\left\{
  \begin{array}{ll}
 D_{*0}^{\frac{1}{2}}y(x)=\displaystyle \frac{\sqrt{\pi}}{2}-\frac{1}{4}+\frac{1}{2}\displaystyle\int_{0}^{1}y^{2}(t)dt,\\
    y(0)=0.
  \end{array}
\right.
\]
\end{example}
with the non-smooth solution $y(x)=\sqrt{x}$. Using relations (\ref{3})-(\ref{kok}), we obtain the equivalent nonlinear integral equation
\begin{align}
 y(x)=\widetilde{g}(x)+ \int_{0}^{x}(x-t)^{\frac{-1}{2}}\widetilde{K}(y(t))dt,
\end{align}
where,
\begin{equation}
\widetilde{K}y(x)=\frac{1}{\sqrt{\pi}}\int_{0}^{1}y^{2}(z)dz,\ \ \
\widetilde{g}(x)=\displaystyle\frac{\sqrt{x}(\sqrt{\pi}-\frac{1}{2})}{\sqrt{\pi}}.
\end{equation}
From (\ref{3vbbb}), we have
\begin{align}
w(x)=\frac{1}{\sqrt{\pi}}\int_{0}^{1}\bigg(\displaystyle\frac{\sqrt{z}(\sqrt{\pi}-\frac{1}{2})}{\sqrt{\pi}}+\displaystyle\int_{0}^{z}(z-t)^{-\frac{1}{2}}w(t)dt\bigg)dz.
\end{align}
Now, we apply our method with $N=1$. From \textbf{Step 1.}, let
\begin{align}
&w_{1}(x)=w_{0}\widetilde{C}_{1,0}(x)+w_{1}\widetilde{C}_{1,1}(x)=\mathbf{W}_{1}\mathbf{\Phi}(x),\nonumber\\
&\mathbf{W}_{1}=[w_{0},w_{1}],\ \ \ \ \mathbf{\Phi}(x)=[2-3\sqrt{x},\sqrt{x}]^{T}.
\end{align}
From \textbf{Step
2.},
\[
\displaystyle \mathcal{P}= \left[
  \begin{array}{cc}
   \displaystyle \frac{\pi}{8}&\displaystyle 4-\frac{9\pi}{8} \\
   \displaystyle \frac{-\pi}{24}&\displaystyle \frac{3\pi}{8}\\
  \end{array}
\right],\ \ \
\widehat{\mathbf{W}}_{1}=\mathbf{W}_{1}
\mathcal{P}=[\frac{\pi}{8}w_{0}-\frac{\pi}{24}w_{1},(4-\frac{9\pi}{8})w_{0}+\frac{3\pi}{8}w_{1}],
\]
\[
\mathcal{C}=[0,0],\ \ \
\widetilde{\mathbf{G}}=[\frac{2\sqrt{\pi}-1}{8\sqrt{\pi}},\frac{6\sqrt{\pi}-3}{8\sqrt{\pi}}],\
\ \
\mathcal{G}=\widetilde{\mathbf{G}}\mathcal{P}=[0,1-\frac{1}{2\sqrt{\pi}}],
\]
\[
\mathcal{H}(z)=\frac{1}{\sqrt{\pi}}\bigg(\big(4\sqrt{x}-\frac{3\pi\sqrt{x}}{2}+\frac{\pi}{4}\big)w_{1}+\big(\frac{\pi\sqrt{x}}{2}-\frac{\pi}{12}\big)w_{2}+\sqrt{x}-\frac{\sqrt{x}}{2\sqrt{\pi}} \bigg)^{2}.
\]
From \textbf{Step 3.}, we obtain
we obtain
\begin{equation}
\mathbf{f}_{i}(\mathbf{W}_{1})=\mathbf{W}_{1}\mathbf{\Phi}(\widehat{x}_{i})-\int_{0}^{1}\mathcal{H}(z)dz=0,
\end{equation}
consequently,
\begin{align}
 \mathds{F}_{1}(\mathbf{W}_{1})&=\big[ \mathbf{f}_{0}(\mathbf{W}_{1}),\mathbf{f}_{1}(\mathbf{W}_{1})\big]\equiv 0,
\end{align}
with the collocation points $\widehat{x}_{0}=\frac{3}{5}+\frac{\sqrt{6}}{10},
\widehat{x}_{1}=\frac{3}{5}-\frac{\sqrt{6}}{10}$. From
\textbf{Step 4.}, Newton's iterative method with initial guess $\mathbf{W}_{1,0}=-\widetilde{\mathbf{G}}$ gives
\begin{align}
\mathbf{W}_{1,1}&=\left[
          \begin{array}{c}
           0.06109105203159421258866747384762992484681\\
           0.1832731560947826377660024215428897745405
          \end{array}
        \right]
,\nonumber\\
&\vdots\nonumber\\
\mathbf{W}_{1,6}&=\left[
          \begin{array}{c}
          0.07052369794346953586850993144509657323068\\
          0.2115710938304086076055297943352897196920
          \end{array}
        \right],
\end{align}
where $\Vert \mathds{F}_{1}(\mathbf{W}_{1,6})\Vert \leq  10^{-40}$. Finally, we obtain an approximate solution as
\[
y_{1}(x)=\widetilde{g}(x)+\mathbf{W}_{1} \mathcal{P} \mathbf{\Phi}(x)=\sqrt{x}+12.5664\times 10^{-40}.
\]
\begin{table}[ht]
\centering
\caption{The exact and approximate solutions at selected points for $N=4$, $\nu=1/2$ for Example \ref{DW}}
\label{pppoo}
\begin{tabular}{cc|c}
\hline\noalign{\smallskip}
x & & Numerical Results\\
\noalign{\smallskip}\hline\noalign{\smallskip}
0.1 &App. sol. &0.1316227766016837933199889354443271853503\\
     &Exa. sol&0.1316227766016837933199889354443271853372\\
     \hline
0.3 &App. sol.&0.4643167672515498340370909348402406402096\\
&Exa. sol&0.4643167672515498340370909348402406401858 \\
     \hline
0.5&App. sol. & 0.8535533905932737622004221810524245196736\\
&Exa. sol&0.8535533905932737622004221810524245196424 \\
     \hline
0.7 & App. sol. &1.285662018573852883584720418049631242609\\
&Exa. sol&1.285662018573852883584720418049631242575 \\
     \hline
0.9 & App. sol.& 1.753814968245462419639701256996834004134\\
&Exa. sol&1.753814968245462419639701256996834004104\\
\noalign{\smallskip}\hline
\end{tabular}
\end{table}
\begin{example}\label{DW}
Consider the problem
\[
\left\{
  \begin{array}{ll}
 D_{*0}^{\frac{1}{2}}y(x)=\displaystyle \frac{2\sqrt{x}}{\sqrt{\pi}}+\frac{3x\sqrt{\pi}}{4}-\frac{9}{10}+\displaystyle\int_{0}^{1}y(t)dt,\\
    y(0)=0,
  \end{array}
\right.
\]
with the non-smooth solution $y(x)=x^{\frac{3}{2}}+x$. The comparison between approximate and exact solution for $N=4$ and $\nu=1/2$ are reported in Table \ref{pppoo} and Fig. \ref{AWqq}. The absolute error obtained between the approximate solutions and the exact solution is plotted in Fig. \ref{AWqq}.  Table \ref{pppoo1} shows the absolute error of our method for $N=4,\ \nu=\frac{1}{2},1$,  Chebyshev-wavelet method \cite{Setia}, and Laguerre-collocation method \cite{Bayram}. The obtained results show that the approximate solution has a good agreement with the exact solution by using the fractional Chelyshkov polynomials.
\end{example}

\begin{figure}[h]
\begin{center}$
\begin{array}{cc}
  \includegraphics[width=0.5\linewidth]{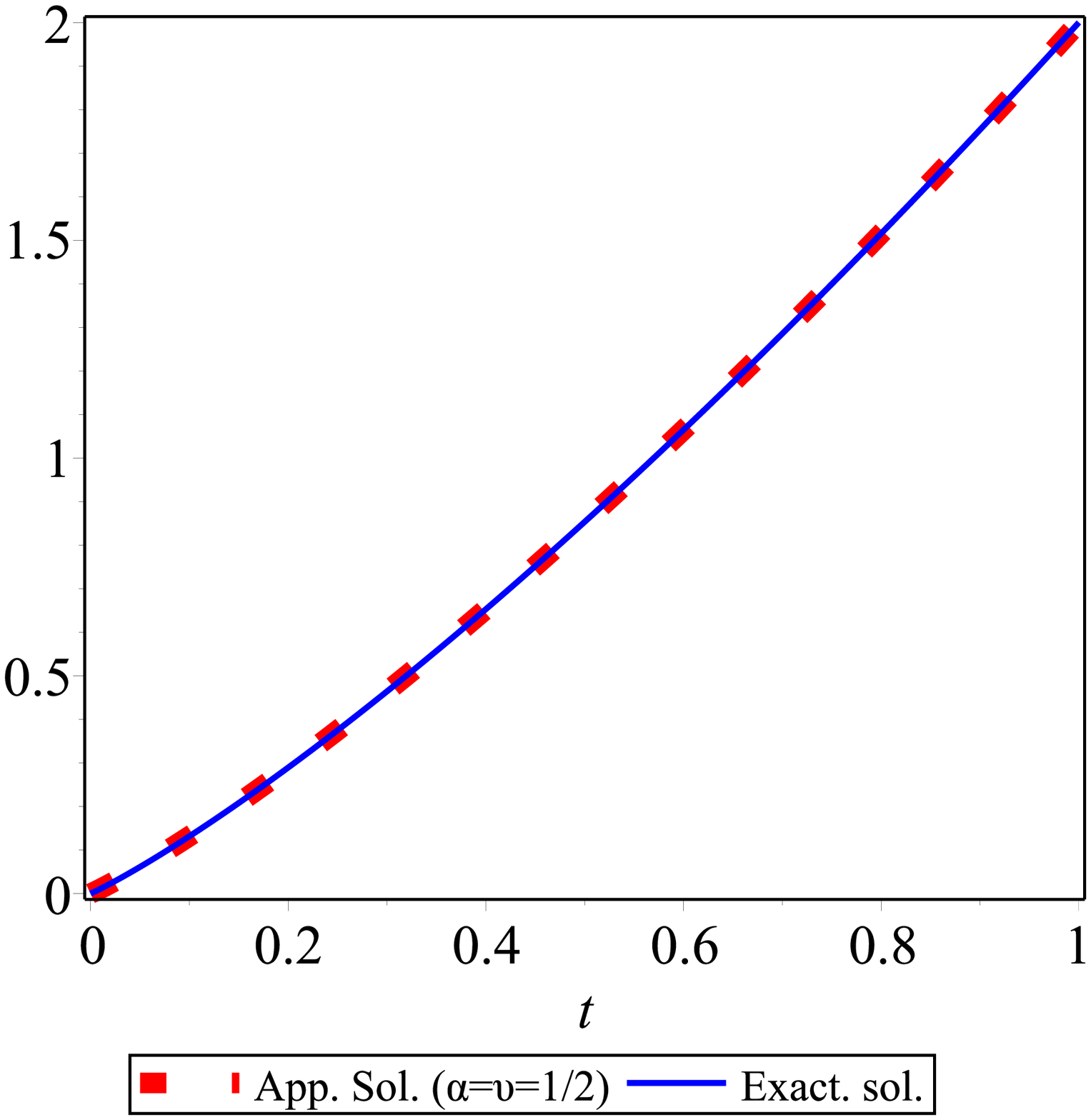}
  \includegraphics[width=0.5\linewidth]{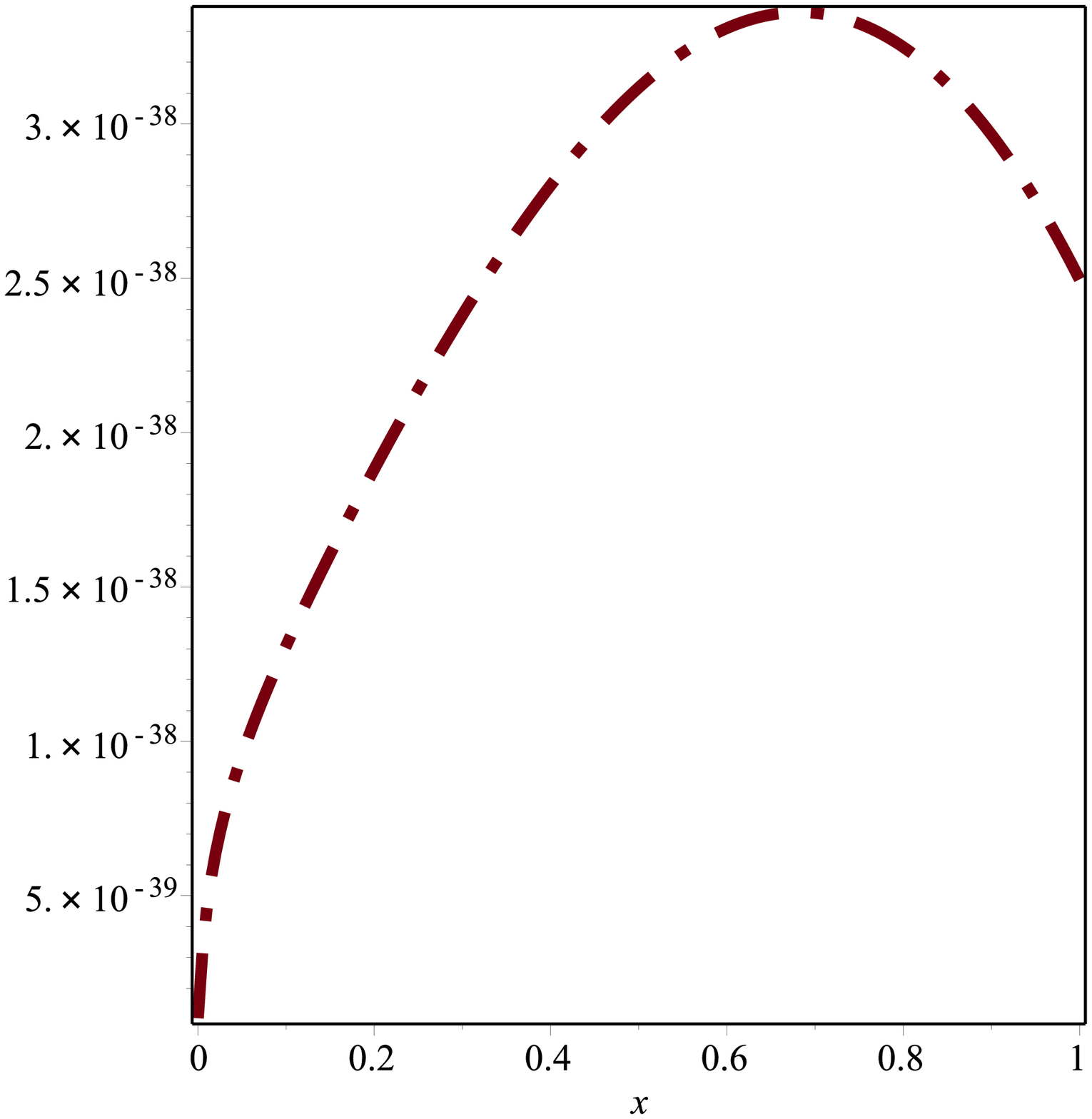}
\end{array}$
\end{center}
\caption{
The plot of the exact and approximate solution (left), the absolute error between the exact and approximation solution with $\nu=\frac{1}{2}$ (right) for Example \ref{DW}.}
\label{AWqq}
\end{figure}

\begin{table}[h]
\centering
\caption{The numerical results of Laguerre-collocation method \cite{Bayram}, Chebyshev-wavelet method \cite{Setia} and our method with  $N=4$ for Example \ref{DW}}
\label{pppoo1}
\begin{tabular}{c|c|c|c|c}
\hline\noalign{\smallskip}
x & Ref. \cite{Bayram} & Ref. \cite{Setia}&Our Method&Our Method\\
  &with $N=7$ & with $k=M=4$ &with $\nu=1$& with $\nu=1/2$\\
\noalign{\smallskip}\hline\noalign{\smallskip}
0.1 &9.5e-5 &1.67722e-3&1.31000e-4&1.31000e-38\\
0.3 & 2.4e-4&1.78323e-3&5.55703e-4&2.38000e-38 \\
0.5 & 6.7e-5&2.04661e-3&8.97348e-4&3.12000e-38 \\
0.7 &2.2e-4&2.33798e-3&7.88172e-5&3.40000e-38 \\
0.9 & 2.2e-4&2.58503e-3&5.16739e-4&3.00000e-38\\
\hline
CPU-Time&Not Reported&Not Reported&2.527s&2.449s\\\hline
$L^2$-error&3.8e-4&2.11330e-3&2.14371e-3&2.16e-38\\
\noalign{\smallskip}\hline
\end{tabular}
\end{table}

 \begin{example}\label{EEE2}
Consider the problem
\begin{equation}\label{KOK}
\left\{
  \begin{array}{ll}
    D_{*0}^{\alpha}y(x)=\displaystyle 1-\frac{x}{4}+\displaystyle\int_{0}^{1}xty^{2}(t)dt, \\
   y(0)=0.
  \end{array}
\right.
\end{equation}
The exact solution is $y(x)=x$ for $\alpha=1$. The approximate solution for various $\alpha=\nu=\frac{1}{4},\frac{1}{2},\frac{3}{4},1$ and $N=8$ are plotted in Fig. \ref{ETT}.  The numerical solutions converge to the solution of problem (\ref{KOK}) for $\alpha=1$ as $\alpha\rightarrow 1$.  In Table \ref{m34}, we report the  $L^2$-error obtained using proposed method and CAS wavelet basis \cite{Saeedi}, Chebyshev wavelet basis \cite{Zhu}, and rationalized Haar functions \cite{Rahimi} in the case of $\alpha= 1$.
 \end{example}

 \begin{figure}[h]
 \centering
  \includegraphics[width=0.7\textwidth]{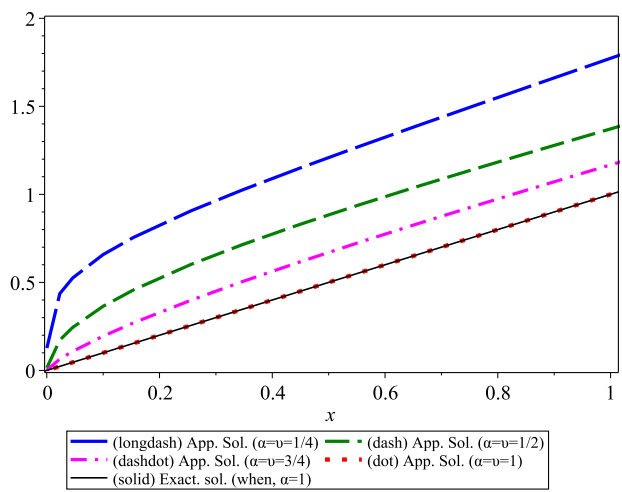}
\caption{The behavior of approximate solutions for various values of $\alpha$ along with the exact solution for Example \ref{EEE2}.}
\label{ETT}
\end{figure}

\begin{table}
\caption{The $L^{2}$-error comparison results between the our method and results of Refs. \cite{Zhu}, \cite{Saeedi} and \cite{Rahimi} in Example \ref{EEE2}}
\label{m34}
\begin{tabular}{lllll}
\hline\noalign{\smallskip}
 & Chebyshev wavelet  & CAS wavelets  & Rationalized Haar & Our method \\
  & \cite{Zhu} &\cite{Saeedi} & function \cite{Rahimi}&($\nu=1$)  \\
 &(k=5, M=2)& (k=4, M=1)  & (N=3, M=8)  &(m=6, N=2)\\\hline
Number of basis&   32&  48&24&2  \\
\noalign{\smallskip}\hline\noalign{\smallskip}
$L^{2}$-error & 1.1645e-9 &1.6745e-5&1.3561e-4&2.0525e-40\\\hline
CPU-Time&-&-&-&0.202s\\
\hline
\end{tabular}
\end{table}
\begin{table}[h]
\caption{The $L^{2}$-error with different basis for Example \ref{Typ}}
\label{m3}
\begin{tabular}{lllll}
\hline\noalign{\smallskip}
 & Chebyshev wavelet\cite{Zhu} & CAS wavelet \cite{Saeedi} &Alpert wavelets \cite{Hag}  \\
&  (k=5, M=2)&(k=4, M=1)&(r=4, J=6) \\
\noalign{\smallskip}\hline\noalign{\smallskip}
$\Vert E \Vert_{N}$ & $2.3374e-7$ &$5.3445e-6$&$1.7295e-5$\\\hline
CPU-Time&-&-&10.151s\\
\hline\noalign{\smallskip}
&Fractional-order &Our method  & Our method &   \\
&Lagrange polynomials \cite{Kumar} &(m=6, N=4)& (m=6, N=4) &   \\
&(n=4,\ $\nu=\frac{1}{2}$)& ($\nu=1$) &($\nu=\frac{1}{2}$) \\
\noalign{\smallskip}\hline\noalign{\smallskip}
 $\Vert E \Vert_{N}$&1.08395e-14& 1.5591e-3 &$1.6792e-39$\\\hline
CPU-Time&-&0.764s&0.765s\\
\noalign{\smallskip}\hline
\end{tabular}
\end{table}
\begin{figure}[h]
\begin{center}$
\begin{array}{cc}
\includegraphics[width=0.5\linewidth]{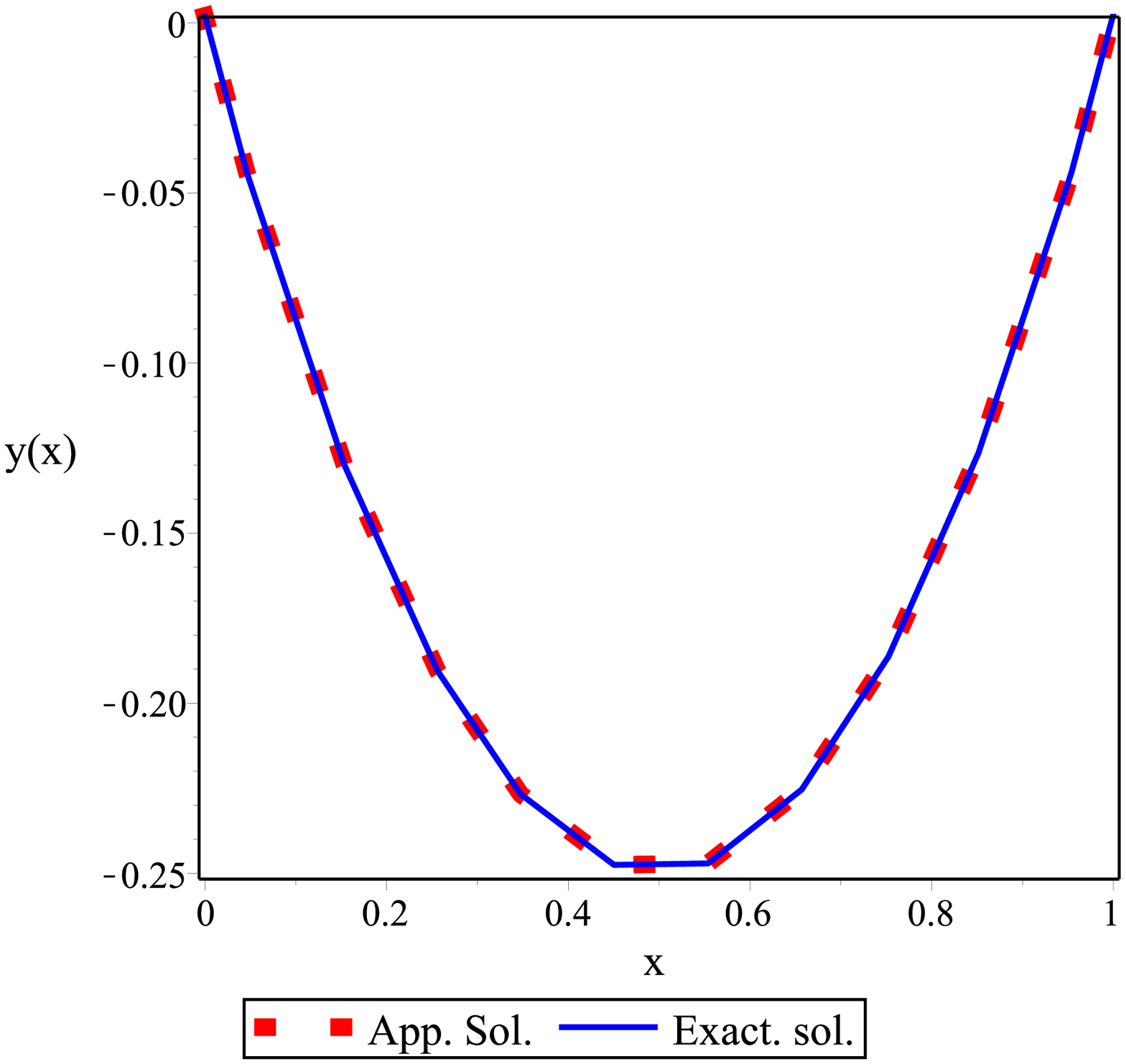}
\includegraphics[width=0.5\linewidth]{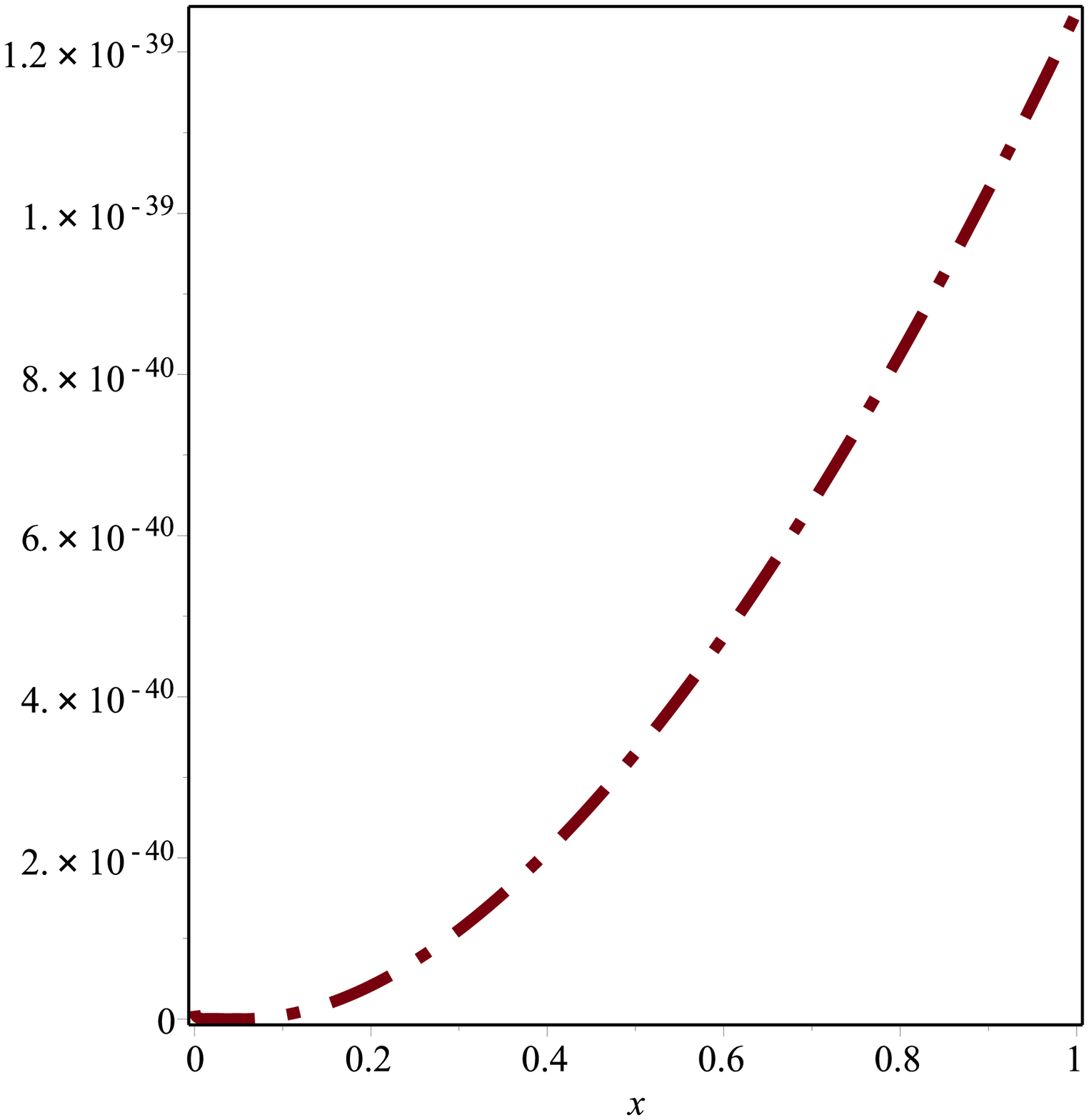}
\end{array}$
\end{center}
\caption{Plot of the exact and approximate solution (left), absolute error between the exact and approximation solution (right) with $\nu=\frac{1}{2} $ for Example \ref{Typ}.}
\label{Exp3}
\end{figure}

\begin{example}\label{Typ}
Consider the following integro-differential equation
\[
\left\{
  \begin{array}{ll}
 D_{*0}^{\frac{1}{2}}y(x)=\displaystyle\frac{1}{\Gamma(1/2)}\left(\frac{8}{3}x^{3/2}-2x^{1/2}\right)-\frac{x}{1260}+\displaystyle\int_{0}^{1}xty^{4}(t)dt,\\
    y(0)=0.
  \end{array}
\right.
\]
The exact solution is $y(x)=x^{2}-x$. Table \ref{m3} displays a comparison between our method and the CAS wavelet method \cite{Saeedi}, Chebyshev wavelet method \cite{Zhu} and Alpert multi-wavelets method \cite{Hag}.
Fig.  \ref{Exp3} shows the good agreement between the exact solution and approximate solution.
The obtained results show that our method by using a small number of fractional Chelyshkov basis produces a more accurate result compared to other basis functions.
\end{example}

\begin{example}\label{Typp}
Consider the problem
\[
\left\{
  \begin{array}{ll}
 D_{*0}^{\frac{1}{2}}y(x)=g(x)+\displaystyle\int_{0}^{1}sin(x+t)y^{2}(t)dt,\\
    y(0)=0,
  \end{array}
\right.
\]
with the non-smooth solution $y(x)= x^{\frac{1}{2}}-\frac{1}{3!}x^{\frac{3}{2}}+\frac{1}{5!}x^{\frac{5}{2}}$.
 Table \ref{TT11}, illustrates the $L^{2}$-errors obtained by our method for different values of $N$ and $\nu$ with $m=10$
  which are also shown in Fig. \ref{ETT44}.
  The obtained results confirm the efficiency and convergence of the method.
It can be seen that a significant improvement in the rate of
convergence of the method is obtained using the fractional Chebyshev
basis functions. The semi-log representation in Fig. \ref{ETT44}
shows the linear variations of the errors versus the degree of
approximation in case of $\nu=\frac{1}{2}$. This is so-called
exponential convergence or spectral accuracy of the collocation
methods that have been recovered in the proposed method for the
problems with non-smooth solutions.
 The absolute error at some selected points
with $N=10$ is reported in Table \ref{pppoou}.
\end{example}
\begin{table}[h]
\begin{center}
\caption{The $L^{2}$-errors for different values of $\nu$ and $N$
for Example \ref{Typp}}
 \begin{tabular}{c|cccccccc}
 \hline\noalign{\smallskip}
N &2&4 &6& 8&10   \\\hline
$\nu=1/4$&4.5944e-02&1.4297e-03&3.0806e-05&6.3394e-07&7.9853e-08\\
$\nu=1/2$&9.1157e-03&1.4795e-05&6.4717e-07&1.6866e-10&1.1612e-10\\
$\nu=3/4$&1.3978e-02&3.3550e-03&1.1277e-03&6.1575e-04&3.4911e-04\\
$\nu=1$ &2.1852e-02&6.0291e-03&2.9305e-03&1.5718e-03&8.4281e-04\\\hline
N &12 &14& 16&18&20   \\\hline
$\nu=1/4$&2.8281e-09&9.0679e-11&1.1469e-11&3.1578e-13&9.2724e-15\\
$\nu=1/2$&3.4954e-13 &9.8229e-15&4.5248e-17&4.6602e-19&2.7650e-21\\
$\nu=3/4$&1.9716e-04 &1.5580e-04&1.0786e-04&7.0349e-05&5.6620e-05\\
$\nu=1$&5.4258e-04 &4.5238e-04&3.7783e-04&2.8438e-04&2.0093e-04\\
\hline
\noalign{\smallskip}
\end{tabular}
\label{TT11}
\end{center}
\end{table}
\begin{figure}[h]
 \centering
  \includegraphics[width=0.8\textwidth]{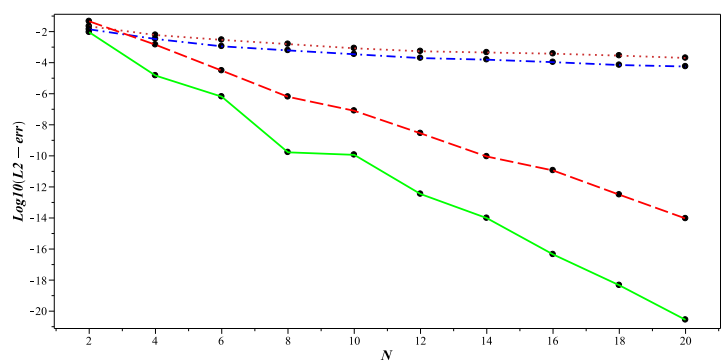}
\caption{The $L^{2}$-error for different values of $N$ with $\nu$=1/4 (dashed-lines), $\nu=1/2$ (solid-line), $\nu=3/4$ (dashed-dotted-lines) and $\nu=1$ (dotted-lines) for Example \ref{Typp}.}
\label{ETT44}
\end{figure}
\begin{table}[h]
\centering
\caption{Absolute errors at some selected points with $m=N=10$ for Example \ref{Typp}}
\label{pppoou}
\begin{tabular}{c|c|c|c|c}
\hline\noalign{\smallskip}
x & $\nu=1/4$ & $\nu=1/2$& $\nu=3/4$&$\nu=1$\\
\noalign{\smallskip}\hline\noalign{\smallskip}
0.1 &4.2857e-08 &3.1881e-11&4.0770e-4&5.6351e-4\\
0.3 &5.7263e-09&5.7105e-11&3.1744e-5&6.4303e-4 \\
0.5 & 6.7919e-08&4.0053e-11&8.3869e-5&4.3161e-4 \\
0.7 &6.7635e-09&7.2990e-11&1.2539e-5&5.3821e-4 \\
0.9 &3.4723e-08&7.7616e-11&1.9118e-4&6.8182e-4\\
\hline
CPU-Time&8.752s&8.331s&11.325s&4.524s\\
\noalign{\smallskip}\hline
\end{tabular}
\end{table}

\section{Conclusions}
In this paper, a new fractional version of the collocation method based on the fractional Chelyshkov polynomials has been introduced to solve a class of nonlinear fractional integro-differential equations. The operational matrix of fractional integrations with the spectral collocation method is utilized to convert the problem into a system of algebraic equations.
Finally, the proposed method was implemented to solve the fractional integro-differential equations with smooth and non-smooth solutions with accurate results.  The proposed method is computationally simple and the approximate solutions converges to the exact solution of the problem as the number of basis (fractional Chelyshkov polynomials) increases. Numerical examples illustrate that the obtained results are more significant than other existing methods.
\section*{Declarations}
{\bf{\small{Conflicts of Interest}}} All other authors have no conflicts of interest to declare.

\end{document}